\documentclass[twoside,reqno,11pt]{amsart}

\usepackage{latexsym}

\newcommand{\qdn}{\hspace*{-1.5mm}}
\newcommand{\qqdn}{\hspace*{-2.5mm}}
\newcommand{\xqdn}{\hspace*{-5.0mm}}
\newcommand{\xxqdn}{\hspace*{-10mm}}

\newcommand{\fns}{\footnotesize}
\newcommand{\sst}{\scriptstyle}

\newcommand{\fnk}[3]{\left[\qdn\ba{#1}#2\\#3\ea\qdn\right]}

\newcommand{\be}{\begin{equation}}
\newcommand{\ee}{\end{equation}}
\newcommand{\ba}{\begin{array}}
\newcommand{\ea}{\end{array}}
\newcommand{\bmn}{\begin{eqnarray}}
\newcommand{\emn}{\end{eqnarray}}
\newcommand{\bnm}{\begin{eqnarray*}}
\newcommand{\enm}{\end{eqnarray*}}
\newcommand{\bln}{\begin{subequations}}
\newcommand{\eln}{\end{subequations}}

\newtheorem{thm}{Theorem}

\newtheorem{corl}[thm]{Corollary}

\newtheorem{exam}{Example}

\newtheorem{entry}{Entry}

\newcommand{\bbtm}[4]{\bibitem{kn:#1}{#2,}~{#3,}~{#4.}}
\newcommand{\cito}[1]{\cite{kn:#1}}

\begin{document} 
{\fns
\title{Series expansions for $1/\pi^m$ and $\pi^m$}
\author{$^a$Chuanan Wei, $^b$Xiaoxia Wang}
\dedicatory{
$^A$Department of Information Technology\\
  Hainan Medical College, Haikou 571199, China\\
  $^B$Department of Mathematics\\
  Shanghai University, Shanghai 200444, China}
\thanks{\emph{Email addresses}:
      weichuanan78@163.com (C. Wei), xwang913 @126.com (X. Wang)}

\address{ }


\begin{abstract}
By means of the telescoping method, we establish two summation
formulas on sine function. As the special cases of them, several
interesting series expansions for $1/\pi^m$ and $\pi^m$ are given.
\end{abstract}

\maketitle\thispagestyle{empty}
\markboth{Chuanan Wei, Xiaoxia Wang}
         {Series expansions for $1/\pi^m$ and $\pi^m$}

\section{Introduction}

For centuries, the study of $\pi$ attracts many mathematicians. A
lot of interesting series expansions for $1/\pi^2$, $1/\pi$, $\pi$
and $\pi^2$ can be seen in the papers
\cito{adanmchik}-\cito{zudilin-c}. Inspired by these work, we shall
explore series expansions for $1/\pi^m$ and $\pi^m$ in terms of the
telescoping method.

 For a complex number $x$ and an integer $n$, define the shifted
factorial by
 \bnm
 (x)_n=
\begin{cases}
\prod_{k=0}^{n-1}(x+k),&\quad\text{when}\quad n>0;\\
1,&\quad\text{when}\quad n=0;\\
 \frac{(-1)^n}{\prod_{k=1}^{-n}(k-x)},&\quad\text{when}\quad n<0.
\end{cases}
 \enm
Recall that the function $\Gamma(x)$ can be defined by Euler's
integral:
\[\Gamma(x)=\int_{0}^{\infty}t^{x-1}e^{-t}dt\:\:\text{with}\:\:Re(x)>0.\]
Then we have the following three relations:
 \bnm
 &&\Gamma(x+n)=\Gamma(x)(x)_n,                   \\
 &&\Gamma(x)\Gamma(1-x)=\frac{\pi}{\sin(\pi x)}, \\
&&\lim_{n\to\infty}\frac{\Gamma(n+x)}{\Gamma(n+y)}n^{y-x}=1,
  \enm
which will frequently be utilized without indication in this paper.
For simplifying the expressions, we shall use the two notations:
 \bnm
&&\xqdn\quad\!\fnk{cccc}{a,&b,&\cdots,&c}{\alpha,&\beta,&\cdots,&\gamma}_n
=\frac{(a)_n(b)_n\cdots(c)_n}{(\alpha)_n(\beta)_n\cdots (\gamma)_n},\\
&&\xqdn\Gamma\fnk{cccc}{a,&b,&\cdots,&c}{\alpha,&\beta,&\cdots,&\gamma}\:\:\,
=\frac{\Gamma(a)\Gamma(b)\cdots\Gamma(c)}{\Gamma(\alpha)\Gamma(\beta)\cdots
\Gamma(\gamma)}.
  \enm

For a complex sequence $\{\tau_k\}_{k\in \mathbb{Z}}$, define
 the difference operator by
\[\nabla\tau_k=\tau_k-\tau_{k-1}.\]
Then the telescoping method can offer the summation:
 \bmn\label{telescoping-summation}
\:\qquad\sum_{k=0}^{n}\nabla\tau_k=\tau_n-\tau_{-1}.
 \emn

The structure of the paper is arranged as follows. We shall found a
summation formula on sine function which includes several series
expansions for $1/\pi^m$ in section 2. And a summation formula on
sine function which implies several series expansions for $\pi^m$
will be derived in section 3.

\section{Series expansions for $1/\pi^m$}

\begin{thm}\label{thm-a} For $m$ complex numbers $\{x_i\}_{i=1}^m$ and $3m$ integers $\{p_i,q_i,r_i\}_{i=1}^m$
with $\min\{r_i, p_i+q_i-r_i+1\}\geq0$, there holds:
 \bnm
 \quad\frac{\prod_{i=1}^m\sin(\pi x_i)}{\pi^m}&&\xqdn=\:
  \sum_{k=0}^{\infty}\prod_{i=1}^m\frac{(x_i)_{k+p_i}(1-x_i)_{k+q_i}}{(k+r_i)!(k+p_i+q_i-r_i+1)!}\\
  &&\xqdn\times\:\Big\{\prod_{i=1}^m(k+x_i+p_i)(k-x_i+q_i+1)-\prod_{i=1}^m(k+r_i)(k+p_i+q_i-r_i+1)\Big\}\\
&&\xqdn+\:\prod_{i=1}^m\frac{(x_i)_{p_i}(1-x_i)_{q_i}}{r_i!(p_i+q_i-r_i+1)!}\,r_i(p_i+q_i-r_i+1).
 \enm
\end{thm}

\begin{proof}
Letting
\[\tau_k=\prod_{i=1}^m\frac{(x_i)_{k+p_i+1}(1-x_i)_{k+q_i+1}}{(k+r_i)!(k+p_i+q_i-r_i+1)!},\]
Then we have
 \bnm
\nabla\tau_k&&\xqdn=\:\prod_{i=1}^m\frac{(x_i)_{k+p_i+1}(1-x_i)_{k+q_i+1}}{(k+r_i)!(k+p_i+q_i-r_i+1)!}
-\prod_{i=1}^m\frac{(x_i)_{k+p_i}(1-x_i)_{k+q_i}}{(k+r_i-1)!(k+p_i+q_i-r_i)!}\\
 &&\xqdn=\:\prod_{i=1}^m\frac{(x_i)_{k+p_i}(1-x_i)_{k+q_i}}{(k+r_i)!(k+p_i+q_i-r_i+1)!}\\
 &&\xqdn\times\:\Big\{\prod_{i=1}^m(k+x_i+p_i)(k-x_i+q_i+1)-\prod_{i=1}^m(k+r_i)(k+p_i+q_i-r_i+1)\Big\}.
 \enm
Substituting the expressions of $\tau_k$ and $\nabla\tau_k$ into
\eqref{telescoping-summation}, we obtain the terminating summation
formula:
 \bnm
 &&\sum_{k=0}^{n}\prod_{i=1}^m\frac{(x_i)_{k+p_i}(1-x_i)_{k+q_i}}{(k+r_i)!(k+p_i+q_i-r_i+1)!}\\
 &&\:\times\:\Big\{\prod_{i=1}^m(k+x_i+p_i)(k-x_i+q_i+1)-\prod_{i=1}^m(k+r_i)(k+p_i+q_i-r_i+1)\Big\}\\
 &&\:=\:\prod_{i=1}^m\frac{(x_i)_{n+p_i+1}(1-x_i)_{n+q_i+1}}{(n+r_i)!(n+p_i+q_i-r_i+1)!}\\
 &&\:-\:\prod_{i=1}^m\frac{(x_i)_{p_i}(1-x_i)_{q_i}}{r_i!(p_i+q_i-r_i+1)!}\,r_i(p_i+q_i-r_i+1).
 \enm
The case $n\to\infty$ of it reads as
  \bnm
 &&\sum_{k=0}^{\infty}\prod_{i=1}^m\frac{(x_i)_{k+p_i}(1-x_i)_{k+q_i}}{(k+r_i)!(k+p_i+q_i-r_i+1)!}\\
 &&\:\times\:\Big\{\prod_{i=1}^m(k+x_i+p_i)(k-x_i+q_i+1)-\prod_{i=1}^m(k+r_i)(k+p_i+q_i-r_i+1)\Big\}\\
 &&\:=\:\prod_{i=1}^m\lim_{n\to\infty}\frac{(x_i)_{n+p_i+1}(1-x_i)_{n+q_i+1}}{(n+r_i)!(n+p_i+q_i-r_i+1)!}\\
 &&\:-\:\prod_{i=1}^m\frac{(x_i)_{p_i}(1-x_i)_{q_i}}{r_i!(p_i+q_i-r_i+1)!}\,r_i(p_i+q_i-r_i+1).
 \enm
Considering that
 \bnm
 &&\qqdn\xqdn\lim_{n\to\infty}\frac{(x_i)_{n+p_i+1}(1-x_i)_{n+q_i+1}}{(n+r_i)!(n+p_i+q_i-r_i+1)!}\\
 &&\qqdn\xqdn\:\:\:=\:\:\frac{1}{\Gamma(x_i)\Gamma(1-x_i)}
 \lim_{n\to\infty}\Gamma\fnk{cccc}{n+x_i+p_i+1,n-x_i+q_i+2}{n+r_i+1,n+p_i+q_i-r_i+2}\\
 &&\qqdn\xqdn\:\:\:=\:\:\frac{\sin(\pi x_i)}{\pi}
 \lim_{n\to\infty}\Gamma\fnk{cccc}{n+x_i+p_i+1}{n+r_i+1}n^{r_i-p_i-x_i}\\
 &&\qqdn\xqdn\:\:\:\times\:\:\lim_{n\to\infty}\Gamma\fnk{cccc}{n-x_i+q_i+2}{n+p_i+q_i-r_i+2}n^{p_i-r_i+x_i}\\
&&\qqdn\xqdn\:\:\:=\:\:\frac{\sin(\pi x_i)}{\pi},
 \enm
we completes the proof of Theorem \ref{thm-a}.
\end{proof}

When $m=1$, Theorem \ref{thm-a} reduces to the simple summation
formula on sine function.

\begin{corl}\label{corl-a} For a complex number x and three integers $\{p,q,r\}$
with $\min\{r, p+q-r+1\}\geq0$, there holds:
 \bnm
 \quad\frac{\sin(\pi x)}{\pi(p-r+x)(1+q-r-x)}&&\xqdn=\sum_{k=0}^{\infty}\frac{(x)_{k+p}(1-x)_{k+q}}
 {(k+r)!(k+p+q-r+1)!}\\&&\xqdn+\frac{r(p+q-r+1)}{(p-r+x)(1+q-r-x)}\frac{(x)_p(1-x)_q}{r!(p+q-r+1)!}.
 \enm
\end{corl}

We point out that the case $r=0$ of Corollary \ref{corl-a} can be
covered by the main theorem of Liu \cito{liu-a}. Eight series
expansions for $1/\pi$ with three free parameters
 from this corollary are displayed as follows.

\begin{exam}[$x=1/2$ in Corollary \ref{corl-a}]\label{exam-a}
  \bnm
 \quad\frac{4}{\pi(2p-2r+1)(2q-2r+1)}&&\xqdn=\sum_{k=0}^{\infty}\frac{(1/2)_{k+p}(1/2)_{k+q}}
 {(k+r)!(k+p+q-r+1)!}\\&&\xqdn+\frac{4r(p+q-r+1)}{(2p-2r+1)(2q-2r+1)}\frac{(1/2)_p(1/2)_q}{r!(p+q-r+1)!}.
 \enm
\end{exam}

\begin{exam}[$x=1/6$ in Corollary \ref{corl-a}]\label{exam-b}
  \bnm
 \quad\frac{18}{\pi(6p-6r+1)(6q-6r+5)}&&\xqdn=\sum_{k=0}^{\infty}\frac{(1/6)_{k+p}(5/6)_{k+q}}
 {(k+r)!(k+p+q-r+1)!}\\&&\xqdn+\frac{36r(p+q-r+1)}{(6p-6r+1)(6q-6r+5)}\frac{(1/6)_p(5/6)_q}{r!(p+q-r+1)!}.
 \enm
\end{exam}

\begin{exam}[$x=1/4$ in Corollary \ref{corl-a}]\label{exam-c}
  \bnm
 \quad\frac{8\sqrt{2}}{\pi(4p-4r+1)(4q-4r+3)}&&\xqdn=\sum_{k=0}^{\infty}\frac{(1/4)_{k+p}(3/4)_{k+q}}
 {(k+r)!(k+p+q-r+1)!}\\&&\xqdn+\frac{16r(p+q-r+1)}{(4p-4r+1)(4q-4r+3)}\frac{(1/4)_p(3/4)_q}{r!(p+q-r+1)!}.
 \enm
\end{exam}

\begin{exam}[$x=1/3$ in Corollary \ref{corl-a}]\label{exam-d}
  \bnm
 \quad\frac{9\sqrt{3}}{2\pi(3p-3r+1)(3q-3r+2)}&&\xqdn=\sum_{k=0}^{\infty}\frac{(1/3)_{k+p}(2/3)_{k+q}}
 {(k+r)!(k+p+q-r+1)!}\\&&\xqdn+\frac{9r(p+q-r+1)}{(3p-3r+1)(3q-3r+2)}\frac{(1/3)_p(2/3)_q}{r!(p+q-r+1)!}.
 \enm
\end{exam}

\begin{exam}[$x=1/10$ in Corollary \ref{corl-a}]\label{exam-e}
  \bnm
 \qqdn\frac{25(\sqrt{5}-1)}{\pi(10p-10r+1)(10q-10r+9)}&&\xqdn=\sum_{k=0}^{\infty}\frac{(1/10)_{k+p}(9/10)_{k+q}}
 {(k+r)!(k+p+q-r+1)!}\\&&\xqdn+\frac{100r(p+q-r+1)}{(10p-10r+1)(10q-10r+9)}\frac{(1/10)_p(9/10)_q}{r!(p+q-r+1)!}.
 \enm
\end{exam}

\begin{exam}[$x=3/10$ in Corollary \ref{corl-a}]\label{exam-f}
  \bnm
\qqdn\frac{25(\sqrt{5}+1)}{\pi(10p-10r+3)(10q-10r+7)}&&\xqdn=\sum_{k=0}^{\infty}\frac{(3/10)_{k+p}(7/10)_{k+q}}
 {(k+r)!(k+p+q-r+1)!}\\&&\xqdn+\frac{100r(p+q-r+1)}{(10p-10r+3)(10q-10r+7)}\frac{(3/10)_p(7/10)_q}{r!(p+q-r+1)!}.
 \enm
\end{exam}

\begin{exam}[$x=1/12$ in Corollary \ref{corl-a}]\label{exam-g}
  \bnm
 \:\frac{36(\sqrt{6}-\sqrt{2})}{\pi(12p-12r+1)(12q-12r+11)}&&\xqdn=\sum_{k=0}^{\infty}\frac{(1/12)_{k+p}(11/12)_{k+q}}
 {(k+r)!(k+p+q-r+1)!}\\&&\xqdn+\frac{144r(p+q-r+1)}{(12p-12r+1)(12q-12r+11)}\frac{(1/12)_p(11/12)_q}{r!(p+q-r+1)!}.
 \enm
\end{exam}

\begin{exam}[$x=5/12$ in Corollary \ref{corl-a}]\label{exam-h}
   \bnm
 \qqdn\frac{36(\sqrt{6}+\sqrt{2})}{\pi(12p-12r+5)(12q-12r+7)}&&\xqdn=\sum_{k=0}^{\infty}\frac{(5/12)_{k+p}(7/12)_{k+q}}
 {(k+r)!(k+p+q-r+1)!}\\&&\xqdn+\frac{144r(p+q-r+1)}{(12p-12r+5)(12q-12r+7)}\frac{(5/12)_p(7/12)_q}{r!(p+q-r+1)!}.
 \enm
\end{exam}

Setting $m=2$ and $p_1=p_2=q_1=q_2=r_1=r_2=0$ in Theorem
\ref{thm-a}, we gain the following summation formula on sine
function.

\begin{corl}\label{corl-b} For two complex numbers $x$ and $y$, there holds:
 \bnm
 \quad\frac{\sin(\pi x)\sin(\pi y)}{\pi^2}&&\xqdn=\sum_{k=0}^{\infty}\frac{(x)_k(1-x)_k(y)_k(1-y)_k}
 {k!^2(k+1)!^2}\\&&\xqdn\times\:\big\{(k^2+k)(x+y-x^2-y^2)+xy(1-x)(1-y)\big\}.
 \enm
\end{corl}

Twelve series expansions for $1/\pi^2$
 from Corollary \ref{corl-b} are laid out as follows.

\begin{exam}[$x=y=1/2$ in Corollary \ref{corl-b}]\label{exam-i}
   \bnm
 \frac{2}{\pi^2}=\sum_{k=0}^{\infty}\frac{(1/2)_{k}^4}{k!^2(k+1)!^2}\{k^2+k+1/8\}.
 \enm
\end{exam}

\begin{exam}[$x=y=1/3$ in Corollary \ref{corl-b}]\label{exam-j}
   \bnm
 \frac{27}{16\pi^2}=\sum_{k=0}^{\infty}\frac{(1/3)_{k}^2(2/3)_{k}^2}{k!^2(k+1)!^2}\{k^2+k+1/9\}.
 \enm
\end{exam}

\begin{exam}[$x=y=1/4$ in Corollary \ref{corl-b}]\label{exam-k}
   \bnm
 \:\frac{4}{3\pi^2}=\sum_{k=0}^{\infty}\frac{(1/4)_{k}^2(3/4)_{k}^2}{k!^2(k+1)!^2}\{k^2+k+3/32\}.
 \enm
\end{exam}

\begin{exam}[$x=y=1/6$ in Corollary \ref{corl-b}]\label{exam-l}
   \bnm
 \:\xxqdn\frac{9}{10\pi^2}=\sum_{k=0}^{\infty}\frac{(1/6)_{k}^2(5/6)_{k}^2}{k!^2(k+1)!^2}\{k^2+k+5/72\}.
 \enm
\end{exam}

\begin{exam}[$x=1/2$ and $x=1/6$ in Corollary \ref{corl-b}]\label{exam-m}
   \bnm
 \frac{9}{7\pi^2}=\sum_{k=0}^{\infty}\frac{(1/2)_{k}^2(1/6)_{k}(5/6)_{k}}{k!^2(k+1)!^2}\{k^2+k+5/56\}.
 \enm
\end{exam}

\begin{exam}[$x=1/10$ and $x=3/10$ in Corollary \ref{corl-b}]\label{exam-n}
   \bnm
 \frac{5}{6\pi^2}=\sum_{k=0}^{\infty}\frac{(1/10)_{k}(3/10)_{k}(7/10)_{k}(9/10)_{k}}{k!^2(k+1)!^2}\{k^2+k+63/1000\}.
 \enm
\end{exam}

\begin{exam}[$x=1/12$ and $x=5/12$ in Corollary \ref{corl-b}]\label{exam-o}
   \bnm
 \:\:\frac{18}{23\pi^2}=\sum_{k=0}^{\infty}\frac{(1/12)_{k}(5/12)_{k}(7/12)_{k}(11/12)_{k}}{k!^2(k+1)!^2}\{k^2+k+385/6624\}.
 \enm
\end{exam}

\begin{exam}[$x=1/2$ and $x=1/4$ in Corollary \ref{corl-b}]\label{exam-p}
   \bnm
\qdn\xxqdn\frac{8\sqrt{2}}{7\pi^2}=\sum_{k=0}^{\infty}\frac{(1/2)_{k}^2(1/4)_{k}(3/4)_{k}}{k!^2(k+1)!^2}\{k^2+k+3/28\}.
 \enm
\end{exam}

\begin{exam}[$x=1/4$ and $x=1/6$ in Corollary \ref{corl-b}]\label{exam-q}
   \bnm
 \frac{36\sqrt{2}}{47\pi^2}=\sum_{k=0}^{\infty}\frac{(1/4)_{k}(3/4)_{k}(1/6)_{k}(5/6)_{k}}{k!^2(k+1)!^2}\{k^2+k+15/188\}.
 \enm
\end{exam}

\begin{exam}[$x=1/2$ and $x=1/3$ in Corollary \ref{corl-b}]\label{exam-r}
   \bnm
 \!\xxqdn\frac{18\sqrt{3}}{17\pi^2}=\sum_{k=0}^{\infty}\frac{(1/2)_{k}^2(1/3)_{k}(2/3)_{k}}{k!^2(k+1)!^2}\{k^2+k+2/17\}.
 \enm
\end{exam}

\begin{exam}[$x=1/3$ and $x=1/6$ in Corollary \ref{corl-b}]\label{exam-s}
   \bnm
 \frac{9\sqrt{3}}{13\pi^2}=\sum_{k=0}^{\infty}\frac{(1/3)_{k}(2/3)_{k}(1/6)_{k}(5/6)_{k}}{k!^2(k+1)!^2}\{k^2+k+10/117\}.
 \enm
\end{exam}

\begin{exam}[$x=1/3$ and $x=1/4$ in Corollary \ref{corl-b}]\label{exam-t}
   \bnm
 \:\xqdn\frac{36\sqrt{6}}{59\pi^2}=\sum_{k=0}^{\infty}\frac{(1/3)_{k}(2/3)_{k}(1/4)_{k}(3/4)_{k}}{k!^2(k+1)!^2}\{k^2+k+6/59\}.
 \enm
\end{exam}

Now we begin to display eight series expansions for $1/\pi^m$ from
Theorem \ref{thm-a}.

\begin{corl}[$x_i=1/2$ and $p_i=q_i=r_i=0$ in Theorem \ref{thm-a} with $1\leq i\leq m$]\label{corl-c}
 \bnm
\:\:\xxqdn\frac{1}{\pi^m}=\sum_{k=0}^{\infty}\frac{(1/2)_k^{2m}}{k!^m(k+1)!^m}\big\{(1/2+k)^{2m}-k^m(k+1)^m\big\}.
 \enm
\end{corl}

\begin{corl}[$x_i=1/6$ and $p_i=q_i=r_i=0$ in Theorem \ref{thm-a} with $1\leq i\leq m$]\label{corl-d}
 \bnm
 \quad\frac{1}{(2\pi)^m}=\sum_{k=0}^{\infty}\frac{(1/6)_k^{m}(5/6)_k^{m}}{k!^m(k+1)!^m}\big\{(1/6+k)^{m}(5/6+k)^{m}-k^m(k+1)^m\big\}.
 \enm
\end{corl}

\begin{corl}[$x_i=1/4$ and $p_i=q_i=r_i=0$ in Theorem \ref{thm-a} with $1\leq i\leq m$]\label{corl-e}
 \bnm
 \quad\frac{1}{(\sqrt{2}\pi)^m}=\sum_{k=0}^{\infty}\frac{(1/4)_k^{m}(3/4)_k^{m}}{k!^m(k+1)!^m}\big\{(1/4+k)^{m}(3/4+k)^{m}-k^m(k+1)^m\big\}.
 \enm
\end{corl}

\begin{corl}[$x_i=1/3$ and $p_i=q_i=r_i=0$ in Theorem \ref{thm-a} with $1\leq i\leq m$]\label{corl-f}
 \bnm
 \quad\bigg(\frac{\sqrt{3}}{2\pi}\bigg)^m=\sum_{k=0}^{\infty}\frac{(1/3)_k^{m}(2/3)_k^{m}}{k!^m(k+1)!^m}\big\{(1/3+k)^{m}(2/3+k)^{m}-k^m(k+1)^m\big\}.
 \enm
\end{corl}

\begin{corl}[$x_i=1/10$ and $p_i=q_i=r_i=0$ in Theorem \ref{thm-a} with $1\leq i\leq m$]\label{corl-g}
 \bnm
 \quad\bigg(\frac{\sqrt{5}-1}{4\pi}\bigg)^m=\sum_{k=0}^{\infty}\frac{(1/10)_k^{m}(9/10)_k^{m}}{k!^m(k+1)!^m}\big\{(1/10+k)^{m}(9/10+k)^{m}-k^m(k+1)^m\big\}.
 \enm
\end{corl}

\begin{corl}[$x_i=3/10$ and $p_i=q_i=r_i=0$ in Theorem \ref{thm-a} with $1\leq i\leq m$]\label{corl-h}
 \bnm
 \quad\bigg(\frac{\sqrt{5}+1}{4\pi}\bigg)^m=\sum_{k=0}^{\infty}\frac{(3/10)_k^{m}(7/10)_k^{m}}{k!^m(k+1)!^m}\big\{(3/10+k)^{m}(7/10+k)^{m}-k^m(k+1)^m\big\}.
 \enm
\end{corl}

\begin{corl}[$x_i=1/12$ and $p_i=q_i=r_i=0$ in Theorem \ref{thm-a} with $1\leq i\leq m$]\label{corl-i}
 \bnm
 \quad\bigg(\frac{\sqrt{6}-\sqrt{2}}{4\pi}\bigg)^m=\sum_{k=0}^{\infty}\frac{(1/12)_k^{m}(11/12)_k^{m}}{k!^m(k+1)!^m}\big\{(1/12+k)^{m}(11/12+k)^{m}-k^m(k+1)^m\big\}.
 \enm
\end{corl}

\begin{corl}[$x_i=5/12$ and $p_i=q_i=r_i=0$ in Theorem \ref{thm-a} with $1\leq i\leq m$]\label{corl-j}
 \bnm
 \quad\bigg(\frac{\sqrt{6}+\sqrt{2}}{4\pi}\bigg)^m=\sum_{k=0}^{\infty}\frac{(5/12)_k^{m}(7/12)_k^{m}}{k!^m(k+1)!^m}\big\{(5/12+k)^{m}(7/12+k)^{m}-k^m(k+1)^m\big\}.
 \enm
\end{corl}

With the change of the parameters, Theorem \ref{thm-a} can produce
more series expansions for $1/\pi^m$. The corresponding results will
not be laid out here.
\section{Series expansions for $\pi^m$}

\begin{thm}\label{thm-b} For $m$ complex numbers $\{x_i\}_{i=1}^m$ and $3m$ integers $\{p_i,q_i,r_i\}_{i=1}^m$
with $\min\{p_i, q_i\}\geq0$, there holds:
 \bnm
 \quad\frac{\pi^m}{\prod_{i=1}^m\sin(\pi x_i)}&&\xqdn=\:
  \sum_{k=0}^{\infty}\prod_{i=1}^m\frac{(k+p_i)!(k+q_i)!}{(x_i)_{k+r_i+1}(1-x_i)_{k+p_i+q_i-r_i+2}}\\
  &&\xqdn\times\:\Big\{\prod_{i=1}^m(k+p_i+1)(k+q_i+1)-\prod_{i=1}^m(k+x_i+r_i)(k-x_i+p_i+q_i-r_i+2)\Big\}\\
&&\xqdn+\:\prod_{i=1}^m\frac{p_i!q_i!}{(x_i)_{r_i}(1-x_i)_{p_i+q_i-r_i+1}}.
 \enm
\end{thm}

\begin{proof}
Letting
\[\tau_k=\prod_{i=1}^m\frac{(k+p_i+1)!(k+q_i+1)!}{(x_i)_{k+r_i+1}(1-x_i)_{k+p_i+q_i-r_i+2}},\]
Then we have
 \bnm
\nabla\tau_k&&\xqdn=\:\prod_{i=1}^m\frac{(k+p_i+1)!(k+q_i+1)!}{(x_i)_{k+r_i+1}(1-x_i)_{k+p_i+q_i-r_i+2}}
-\prod_{i=1}^m\frac{(k+p_i)!(k+q_i)!}{(x_i)_{k+r_i}(1-x_i)_{k+p_i+q_i-r_i+1}}\\
 &&\xqdn=\:\prod_{i=1}^m\frac{(k+p_i)!(k+q_i)!}{(x_i)_{k+r_i+1}(1-x_i)_{k+p_i+q_i-r_i+2}}\\
 &&\xqdn\times\:\Big\{\prod_{i=1}^m(k+p_i+1)(k+q_i+1)-\prod_{i=1}^m(k+x_i+r_i)(k-x_i+p_i+q_i-r_i+2)\Big\}.
 \enm
Substituting the expressions of $\tau_k$ and $\nabla\tau_k$ into
\eqref{telescoping-summation}, we get the terminating summation
formula:
 \bnm
 &&\sum_{k=0}^{n}\prod_{i=1}^m\frac{(k+p_i)!(k+q_i)!}{(x_i)_{k+r_i+1}(1-x_i)_{k+p_i+q_i-r_i+2}}\\
 &&\:\times\:\Big\{\prod_{i=1}^m(k+p_i+1)(k+q_i+1)-\prod_{i=1}^m(k+x_i+r_i)(k-x_i+p_i+q_i-r_i+2)\Big\}\\
 &&\:=\:\prod_{i=1}^m\frac{(n+p_i+1)!(n+q_i+1)!}{(x_i)_{n+r_i+1}(1-x_i)_{n+p_i+q_i-r_i+2}}
 -\prod_{i=1}^m\frac{p_i!q_i!}{(x_i)_{r_i}(1-x_i)_{p_i+q_i-r_i+1}}.
 \enm
The case $n\to\infty$ of it reads as
  \bnm
 &&\sum_{k=0}^{\infty}\prod_{i=1}^m\frac{(k+p_i)!(k+q_i)!}{(x_i)_{k+r_i+1}(1-x_i)_{k+p_i+q_i-r_i+2}}\\
 &&\:\times\:\Big\{\prod_{i=1}^m(k+p_i+1)(k+q_i+1)-\prod_{i=1}^m(k+x_i+r_i)(k-x_i+p_i+q_i-r_i+2)\Big\}\\
 &&\:=\:\lim_{n\to\infty}\prod_{i=1}^m\frac{(n+p_i+1)!(n+q_i+1)!}{(x_i)_{n+r_i+1}(1-x_i)_{n+p_i+q_i-r_i+2}}
 -\prod_{i=1}^m\frac{p_i!q_i!}{(x_i)_{r_i}(1-x_i)_{p_i+q_i-r_i+1}}.
 \enm
Considering that
 \bnm
 &&\qqdn\xqdn\lim_{n\to\infty}\prod_{i=1}^m\frac{(n+p_i+1)!(n+q_i+1)!}{(x_i)_{n+r_i+1}(1-x_i)_{n+p_i+q_i-r_i+2}}\\
 &&\qqdn\xqdn\:\:\:=\:\:\Gamma(x_i)\Gamma(1-x_i)
 \lim_{n\to\infty}\Gamma\fnk{cccc}{n+p_i+2,n+q_i+2}{n+x_i+r_i+1,n-x_i+p_i+q_i-r_i+3}\\
 &&\qqdn\xqdn\:\:\:=\:\:\frac{\pi}{\sin(\pi x_i)}
 \lim_{n\to\infty}\Gamma\fnk{cccc}{n+p_i+2}{n+x_i+r_i+1}n^{x_i-p_i+r_i-1}\\
 &&\qqdn\xqdn\:\:\:\times\:\:\lim_{n\to\infty}\Gamma\fnk{cccc}{n+q_i+2}{n-x_i+p_i+q_i-r_i+3}n^{p_i-x_i-r_i+1}\\
&&\qqdn\xqdn\:\:\:=\:\:\frac{\pi}{\sin(\pi x_i)},
 \enm
we finishes the proof of Theorem \ref{thm-b}.
\end{proof}

When $m=1$, Theorem \ref{thm-b} reduces to the simple summation
formula on sine function.

\begin{corl}\label{corl-k} For a complex number x and three integers $\{p,q,r\}$
with $\min\{p, q\}\geq0$, there holds:
 \bnm
 \xqdn\frac{\pi}{{\sst(1+p-r-x)(1+q-r-x)}\sin(\pi x)}&&\xqdn=\sum_{k=0}^{\infty}\frac{(k+p)!(k+q)!}
 {(x)_{k+r+1}(1-x)_{k+p+q-r+2}}\\&&\xqdn
  +\frac{1}{\sst(1+p-r-x)(1+q-r-x)}\frac{p!q!}{(x)_r(1-x)_{p+q-r+1}}.
 \enm
\end{corl}

Twelve series expansions for $\pi$ with three free parameters
 from Corollary \ref{corl-k} are displayed as follows.

\begin{exam}[$x=1/2$ in Corollary \ref{corl-k}]\label{exam-s}
  \bnm
 \xqdn\frac{4\pi}{(2p-2r+1)(2q-2r+1)}&&\xqdn=\sum_{k=0}^{\infty}\frac{(k+p)!(k+q)!}
 {(1/2)_{k+r+1}(1/2)_{k+p+q-r+2}}\\&&\xqdn
  +\frac{4}{(2p-2r+1)(2q-2r+1)}\frac{p!q!}{(1/2)_r(1/2)_{p+q-r+1}}.
 \enm
\end{exam}

\begin{exam}[$x=1/6$ in Corollary \ref{corl-k}]\label{exam-t}
  \bnm
 \xqdn\frac{72\pi}{(6p-6r+5)(6q-6r+5)}&&\xqdn=\sum_{k=0}^{\infty}\frac{(k+p)!(k+q)!}
 {(1/6)_{k+r+1}(5/6)_{k+p+q-r+2}}\\&&\xqdn
  +\frac{36}{(6p-6r+5)(6q-6r+5)}\frac{p!q!}{(1/6)_r(5/6)_{p+q-r+1}}.
 \enm
\end{exam}

\begin{exam}[$x=1/4$ in Corollary \ref{corl-k}]\label{exam-u}
  \bnm
 \qqdn\frac{32\pi}{\sqrt{2}\,(4p-4r+3)(4q-4r+3)}&&\xqdn=\sum_{k=0}^{\infty}\frac{(k+p)!(k+q)!}
 {(1/4)_{k+r+1}(3/4)_{k+p+q-r+2}}\\&&\xqdn
  +\frac{16}{(4p-4r+3)(4q-4r+3)}\frac{p!q!}{(1/4)_r(3/4)_{p+q-r+1}}.
 \enm
\end{exam}

\begin{exam}[$x=1/3$ in Corollary \ref{corl-k}]\label{exam-v}
  \bnm
 \xqdn\frac{18\pi}{\sqrt{3}\,(3p-3r+2)(3q-3r+2)}&&\xqdn=\sum_{k=0}^{\infty}\frac{(k+p)!(k+q)!}
 {(1/3)_{k+r+1}(2/3)_{k+p+q-r+2}}\\&&\xqdn
  +\frac{9}{(3p-3r+2)(3q-3r+2)}\frac{p!q!}{(1/3)_r(2/3)_{p+q-r+1}}.
 \enm
\end{exam}

\begin{exam}[$x=1/10$ in Corollary \ref{corl-k}]\label{exam-w}
  \bnm
 \xxqdn\frac{400\pi}{\sst(\sqrt{5}-1)(10p-10r+9)(10q-10r+9)}&&\xqdn=\sum_{k=0}^{\infty}\frac{(k+p)!(k+q)!}
 {(1/10)_{k+r+1}(9/10)_{k+p+q-r+2}}\\&&\xqdn
  +\frac{100}{\sst(10p-10r+9)(10q-10r+9)}\frac{p!q!}{(1/10)_r(9/10)_{p+q-r+1}}.
 \enm
\end{exam}

\begin{exam}[$x=3/10$ in Corollary \ref{corl-k}]\label{exam-x}
  \bnm
 \xxqdn\frac{400\pi}{\sst(\sqrt{5}+1)(10p-10r+7)(10q-10r+7)}&&\xqdn=\sum_{k=0}^{\infty}\frac{(k+p)!(k+q)!}
 {(3/10)_{k+r+1}(7/10)_{k+p+q-r+2}}\\&&\xqdn
  +\frac{100}{\sst(10p-10r+7)(10q-10r+7)}\frac{p!q!}{(3/10)_r(7/10)_{p+q-r+1}}.
 \enm
\end{exam}

\begin{exam}[$x=1/12$ in Corollary \ref{corl-k}]\label{exam-y}
  \bnm
 \:\qqdn\frac{576\pi}{\sst(\sqrt{6}-\sqrt{2})(12p-12r+11)(12q-12r+11)}&&\xqdn
 =\sum_{k=0}^{\infty}\frac{(k+p)!(k+q)!}{(1/12)_{k+r+1}(11/12)_{k+p+q-r+2}}\\&&\xqdn
  +\frac{144}{\sst(12p-12r+11)(12q-12r+11)}\frac{p!q!}{(1/12)_r(11/12)_{p+q-r+1}}.
 \enm
\end{exam}

\begin{exam}[$x=5/12$ in Corollary \ref{corl-k}]\label{exam-z}
  \bnm
 \qdn\xqdn\frac{576\pi}{\sst(\sqrt{6}+\sqrt{2})(12p-12r+7)(12q-12r+7)}&&\xqdn
 =\sum_{k=0}^{\infty}\frac{(k+p)!(k+q)!}{(5/12)_{k+r+1}(7/12)_{k+p+q-r+2}}\\&&\xqdn
  +\frac{144}{\sst(12p-12r+7)(12q-12r+7)}\frac{p!q!}{(5/12)_r(7/12)_{p+q-r+1}}.
 \enm
\end{exam}

Taking $m=2$ and $p_1=p_2=q_1=q_2=r_1=r_2=0$ in Theorem \ref{thm-b},
we attain the following summation formula on sine function.

\begin{corl}\label{corl-l} For two complex numbers $x$ and $y$, there holds:
 \bnm
 \quad\frac{\pi^2}{\sin(\pi x)\sin(\pi y)}&&\xqdn=\frac{1}{(1-x)(1-y)}+\sum_{k=0}^{\infty}\frac{k!^4}
 {(x)_{k+1}(1-x)_{k+2}(y)_{k+1}(1-y)_{k+2}}\\&&\xqdn\times\:\big\{(k^2+2k)(2-2x-2y+x^2+y^2)+1-xy(2-x)(2-y)\big\}.
 \enm
\end{corl}

 Twelve series expansions for $\pi^2$
 from Corollary \ref{corl-l} are laid out as follows.

\begin{exam}[$x=y=1/2$ in Corollary \ref{corl-l}]\label{exam-aa}
   \bnm
\xqdn\frac{9\pi^2}{32}-\frac{9}{8}=\sum_{k=0}^{\infty}\frac{k!^4}{(3/2)_k^2(5/2)_k^2}\big\{k^2+2k+7/8\big\}.
 \enm
\end{exam}

\begin{exam}[$x=y=1/3$ in Corollary \ref{corl-l}]\label{exam-bb}
   \bnm
\xqdn\frac{50\pi^2}{243}-\frac{25}{72}=\sum_{k=0}^{\infty}\frac{k!^4}{(4/3)_k^2(8/3)_k^2}\big\{k^2+2k+7/9\big\}.
 \enm
\end{exam}

\begin{exam}[$x=y=1/4$ in Corollary \ref{corl-l}]\label{exam-cc}
   \bnm
\frac{49\pi^2}{256}-\frac{49}{288}=\sum_{k=0}^{\infty}\frac{k!^4}{(5/4)_k^2(11/4)_k^2}\big\{k^2+2k+23/32\big\}.
 \enm
\end{exam}

\begin{exam}[$x=y=1/6$ in Corollary \ref{corl-l}]\label{exam-dd}
   \bnm
\xqdn\xxqdn\frac{121\pi^2}{648}-\frac{121}{1800}=\sum_{k=0}^{\infty}\frac{k!^4}{(7/6)_k^2(17/6)_k^2}\big\{k^2+2k+47/72\big\}.
 \enm
\end{exam}

\begin{exam}[$x=1/2$ and $y=1/6$ in Corollary \ref{corl-l}]\label{exam-ee}
   \bnm
\quad\frac{55\pi^2}{272}-\frac{33}{136}=\sum_{k=0}^{\infty}\frac{k!^4}{(3/2)_k(5/2)_k(7/6)_k(17/6)_k}\big\{k^2+2k+111/136\big\}.
 \enm
\end{exam}

\begin{exam}[$x=1/10$ and $y=3/10$ in Corollary \ref{corl-l}]\label{exam-ff}
   \bnm
\quad\frac{61047\pi^2}{325000}-\frac{969}{13000}=\sum_{k=0}^{\infty}\frac{k!^4}{(11/10)_k(13/10)_k(27/10)_k(29/10)_k}\big\{k^2+2k+9031/13000\big\}.
 \enm
\end{exam}

\begin{exam}[$x=1/12$ and $y=5/12$ in Corollary \ref{corl-l}]\label{exam-gg}
   \bnm
\quad\frac{33649\pi^2}{176256}-\frac{437}{4896}=\sum_{k=0}^{\infty}\frac{k!^4}{(13/12)_k(17/12)_k(31/12)_k(35/12)_k}\big\{k^2+2k+18551/24480\big\}.
 \enm
\end{exam}

\begin{exam}[$x=1/2$ and $y=1/4$ in Corollary \ref{corl-l}]\label{exam-hh}
   \bnm
\frac{63\pi^2}{208\sqrt{2}}-\frac{21}{52}=\sum_{k=0}^{\infty}\frac{k!^4}{(3/2)_k(5/2)_k(5/4)_k(11/4)_k}\big\{k^2+2k+43/52\big\}.
 \enm
\end{exam}

\begin{exam}[$x=1/4$ and $y=1/6$ in Corollary \ref{corl-l}]\label{exam-ii}
   \bnm
\qquad\frac{385\pi^2}{1448\sqrt{2}}-\frac{77}{724}=\sum_{k=0}^{\infty}\frac{k!^4}{(5/4)_k(11/4)_k(7/6)_k(17/6)_k}\big\{k^2+2k+499/724\big\}.
 \enm
\end{exam}

\begin{exam}[$x=1/2$ and $y=1/3$ in Corollary \ref{corl-l}]\label{exam-jj}
   \bnm
\xqdn\qqdn\frac{2\pi^2}{5\sqrt{3}}-\frac{3}{5}=\sum_{k=0}^{\infty}\frac{k!^4}{(3/2)_k(5/2)_k(4/3)_k(8/3)_k}\big\{k^2+2k+21/25\big\}.
 \enm
\end{exam}

\begin{exam}[$x=1/3$ and $y=1/6$ in Corollary \ref{corl-l}]\label{exam-kk}
   \bnm
\:\quad\frac{1100\pi^2}{3321\sqrt{3}}-\frac{55}{369}=\sum_{k=0}^{\infty}\frac{k!^4}{(4/3)_k(8/3)_k(7/6)_k(17/6)_k}\big\{k^2+2k+269/369\big\}.
 \enm
\end{exam}

\begin{exam}[$x=1/3$ and $y=1/4$ in Corollary \ref{corl-l}]\label{exam-ll}
   \bnm
\frac{14\pi^2}{29\sqrt{6}}-\frac{7}{29}=\sum_{k=0}^{\infty}\frac{k!^4}{(4/3)_k(8/3)_k(5/4)_k(11/4)_k}\big\{k^2+2k+109/145\big\}.
 \enm
\end{exam}

Now we begin to display eight series expansions for $\pi^m$ from
Theorem \ref{thm-b}.

\begin{corl}[$x_i=1/2$ and $p_i=q_i=r_i=0$ in Theorem \ref{thm-b} with $1\leq i\leq m$]\label{corl-m}
 \bnm
\quad\bigg(\frac{3\pi}{8}\bigg)^m-\bigg(\frac{3}{4}\bigg)^m=\sum_{k=0}^{\infty}\frac{k!^{2m}}{(3/2)_k^m(5/2)_k^m}\big\{(k+1)^{2m}-(1/2+k)^m(3/2+k)^m\big\}.
 \enm
\end{corl}

\begin{corl}[$x_i=1/6$ and $p_i=q_i=r_i=0$ in Theorem \ref{thm-b} with $1\leq i\leq m$]\label{corl-n}
 \bnm
\quad\:\bigg(\frac{55\pi}{108}\bigg)^m-\bigg(\frac{11}{36}\bigg)^m=\sum_{k=0}^{\infty}\frac{k!^{2m}}{(7/6)_k^m(17/6)_k^m}\big\{(k+1)^{2m}-(1/6+k)^m(11/6+k)^m\big\}.
 \enm
\end{corl}

\begin{corl}[$x_i=1/4$ and $p_i=q_i=r_i=0$ in Theorem \ref{thm-b} with $1\leq i\leq m$]\label{corl-o}
 \bnm
\quad\:\bigg(\frac{21\pi}{32\sqrt{2}}\bigg)^m-\bigg(\frac{7}{16}\bigg)^m=\sum_{k=0}^{\infty}\frac{k!^{2m}}{(5/4)_k^m(11/4)_k^m}\big\{(k+1)^{2m}-(1/4+k)^m(7/4+k)^m\big\}.
 \enm
\end{corl}

\begin{corl}[$x_i=1/3$ and $p_i=q_i=r_i=0$ in Theorem \ref{thm-b} with $1\leq i\leq m$]\label{corl-p}
 \bnm
\quad\:\bigg(\frac{20\pi}{27\sqrt{3}}\bigg)^m-\bigg(\frac{5}{9}\bigg)^m=\sum_{k=0}^{\infty}\frac{k!^{2m}}{(4/3)_k^m(8/3)_k^m}\big\{(k+1)^{2m}-(1/3+k)^m(5/3+k)^m\big\}.
 \enm
\end{corl}

\begin{corl}[$x_i=1/10$ and $p_i=q_i=r_i=0$ in Theorem \ref{thm-b} with $1\leq i\leq m$]\label{corl-q}
 \bnm
\quad\:\bigg\{\frac{171\pi}{250(\sqrt{5}-1)}\bigg\}^m-\bigg(\frac{19}{100}\bigg)^m
&&\xqdn=\:\sum_{k=0}^{\infty}\frac{k!^{2m}}{(11/10)_k^m(29/10)_k^m}\\&&\xqdn\times\:\big\{(k+1)^{2m}-(1/10+k)^m(19/10+k)^m\big\}.
 \enm
\end{corl}

\begin{corl}[$x_i=3/10$ and $p_i=q_i=r_i=0$ in Theorem \ref{thm-b} with $1\leq i\leq m$]\label{corl-r}
 \bnm
\quad\:\bigg\{\frac{357\pi}{250(\sqrt{5}+1)}\bigg\}^m-\bigg(\frac{51}{100}\bigg)^m
&&\xqdn=\:\sum_{k=0}^{\infty}\frac{k!^{2m}}{(13/10)_k^m(27/10)_k^m}\\&&\xqdn\times\:\big\{(k+1)^{2m}-(3/10+k)^m(17/10+k)^m\big\}.
 \enm
\end{corl}

\begin{corl}[$x_i=1/12$ and $p_i=q_i=r_i=0$ in Theorem \ref{thm-b} with $1\leq i\leq m$]\label{corl-s}
 \bnm
\quad\:\bigg\{\frac{253\pi}{432(\sqrt{6}-\sqrt{2})}\bigg\}^m-\bigg(\frac{23}{144}\bigg)^m
&&\xqdn=\:\sum_{k=0}^{\infty}\frac{k!^{2m}}{(13/12)_k^m(35/12)_k^m}\\&&\xqdn\times\:\big\{(k+1)^{2m}-(1/12+k)^m(23/12+k)^m\big\}.
 \enm
\end{corl}

\begin{corl}[$x_i=5/12$ and $p_i=q_i=r_i=0$ in Theorem \ref{thm-b} with $1\leq i\leq m$]\label{corl-t}
 \bnm
\quad\:\bigg\{\frac{675\pi}{432(\sqrt{6}+\sqrt{2})}\bigg\}^m-\bigg(\frac{95}{144}\bigg)^m
&&\xqdn=\:\sum_{k=0}^{\infty}\frac{k!^{2m}}{(17/12)_k^m(31/12)_k^m}\\&&\xqdn\times\:\big\{(k+1)^{2m}-(5/12+k)^m(19/12+k)^m\big\}.
 \enm
\end{corl}

With the change of the parameters, Theorem \ref{thm-b} can create
more series expansions for $\pi^m$. The corresponding results will
not be laid out here.

\textbf{Acknowledgments}

The work is supported by the Natural Science Foundations of China
(Nos.11301120, 11201241 and 11201291).


\end{document}